\documentclass[12pt,letterpaper]{amsart}

\newcounter{myenum}

\newcommand{\I}{\operatorname{Im}}
\newcommand{\ds}{\displaystyle}

\newcommand{\eps}{\varepsilon}

\newcommand{\crit}{{\textnormal{c}}}
\newcommand{\crin}{{\textnormal{c},0}}

\usepackage{amssymb}
\usepackage{amsthm}
\usepackage{amsxtra}
\usepackage{graphicx}

\newtheorem{theorem}{Theorem}

\newtheorem{proposition}[theorem]{Proposition}
\newtheorem{lemma}[theorem]{Lemma}
\newtheorem{corollary}[theorem]{Corollary}

\theoremstyle{remark}
\newtheorem{remark}[theorem]{Remark}

\newcommand{\cR}{\mathbb{R}}

\newcommand{\ep}{\epsilon}
\newcommand{\de}{\delta}

\newcommand{\supp}{\mathrm{supp}\,}
\newcommand{\lap}{\triangle}

\DeclareMathOperator{\im}{Im}

\numberwithin{equation}{section}
\numberwithin{theorem}{section}

\ifx\pdfoutput\undefined
  \DeclareGraphicsExtensions{.pstex, .eps}
\else
  \ifx\pdfoutput\relax
    \DeclareGraphicsExtensions{.pstex, .eps}
  \else
    \ifnum\pdfoutput>0
      \DeclareGraphicsExtensions{.pdf}
    \else
      \DeclareGraphicsExtensions{.pstex, .eps}
    \fi
  \fi
\fi

\title[Nonradial 3D cubic NLS]
{Scattering for the non-radial 3D cubic nonlinear Schr\"odinger
equation}

\author{Thomas Duyckaerts}
\address{Universit\'e de Cergy-Pontoise, UMR CNRS 8088, France}
\author{Justin Holmer}
\address{University of California, Berkeley}
\author{Svetlana Roudenko}
\address{Arizona State University}

\begin{document}

\maketitle

\begin{abstract}
Scattering of radial $H^1$ solutions to the 3D focusing cubic
nonlinear Schr\"odinger equation below a mass-energy threshold
$M[u]E[u] < M[Q]E[Q]$ and satisfying an initial mass-gradient bound
$\Vert u_0 \Vert_{L^2} \Vert \nabla u_0 \Vert_{L^2} < \Vert Q
\Vert_{L^2} \Vert \nabla Q \Vert_{L^2}$, where $Q$ is the ground
state, was established in Holmer-Roudenko \cite{HR2}.  In this note,
we extend the result in \cite{HR2} to non-radial $H^1$ data.
For this, we prove a
non-radial profile decomposition involving a spatial translation
parameter.   Then, in the spirit of Kenig-Merle \cite{KM06b}, we control via
momentum conservation the rate of divergence of the spatial
translation parameter and by a convexity argument based on a local
virial identity deduce scattering. An application to the defocusing
case is also mentioned.
\end{abstract}

\section{Introduction}

We consider the Cauchy problem for the cubic focusing nonlinear
Schr\"odinger (NLS) equation on $\mathbb{R}^3$:
\begin{equation}
 \label{E:NLS}
i\partial_t u +\Delta u + |u|^2 \, u=0, \quad (x,t)\in
\mathbb{R}^3\times \mathbb{R},
\end{equation}
\begin{equation}
 \label{E:initialdata}
u(x,0) = u_0 \in H^1(\cR^3).
\end{equation}
\textbf{}It is locally well-posed (e.g., see Cazenave
\cite{Caz-book}).  The equation has 3 conserved quantities; namely,
the mass $M[u]$, energy $E[u]$ and momentum $P[u]$:
$$
M[u] = \int |u(x,t)|^2 \, dx = M[u_0],
$$
$$
E[u] = \frac12\int |\nabla u(x,t)|^2 \, dx - \frac14 \int |u(x,t)|^4
\, dx = E[u_0],
$$
$$
P[u] = \im \int \bar{u}(x,t) \, \nabla u(x,t) \, dx = P[u_0].
$$
The scale-invariant Sobolev norm is $\dot{H}^{1/2}$ and the
scale-invariant Lebesgue norm is $L^3$. Let $u(x,t) = e^{it} Q(x)$;
then $u$ solves \eqref{E:NLS} provided $Q$ solves the nonlinear
elliptic equation
\begin{equation}
 \label{E:Q}
-Q + \Delta Q + |Q|^2 \, Q=0.
\end{equation}
This equation has an infinite number of solutions in $H^1(\cR^3)$.
The solution of minimal mass, hereafter denoted by $Q(x)$, is
positive, radial, exponentially decaying, and is called {\it the
ground state}. For further properties of $Q$, we refer to Weinstein
\cite{W83}, Holmer-Roudenko \cite{HR2}, Cazenave \cite{Caz-book}.

In Holmer-Roudenko \cite[Theorem 1.1]{HR2} (see also Holmer-Roudenko
\cite{HR1}), it was proved that under the condition $M[u] E[u] <
M[Q]E[Q]$, solutions to \eqref{E:NLS}-\eqref{E:initialdata} globally
exist if $u_0$ satisfies
\begin{equation}
 \label{E:gradient}
\Vert u_0 \Vert_{L^2} \Vert \nabla u_0 \Vert_{L^2} < \Vert Q
\Vert_{L^2} \Vert \nabla Q \Vert_{L^2},
\end{equation}
and radial solutions with initial data satisfying \eqref{E:gradient}
scatter in $H^1$ in both time directions.  This means that there
exist $\phi_\pm \in H^1$ such that
$$
\lim_{t\to \pm\infty} \|u(t)-e^{it\Delta}\phi_\pm\|_{H^1} =0.
$$
In this note we extend the scattering result to include non-radial $H^1$ data.

\begin{theorem}
 \label{T:non-radial}
Let $u_0\in H^1$ and let $u$ be the corresponding solution to
\eqref{E:NLS} in $H^1$. Suppose
\begin{equation}
 \label{E:mass_energy_threshold}
M[u]E[u]<M[Q]E[Q].
\end{equation}
If $\|u_0\|_{L^2}\|\nabla u_0\|_{L^2}<\|Q\|_{L^2}\|\nabla
Q\|_{L^2}$, then $u$ scatters in $H^1$.
\end{theorem}

The argument of \cite{HR2} in the radial case followed a strategy
introduced by Kenig-Merle \cite{KM06a} for proving global
well-posedness and scattering for the focusing energy-critical NLS.
The argument begins by contradiction:  suppose the threshold for
scattering is strictly below that claimed.  A profile decomposition
lemma based on concentration compactness principles (and analogous
to that of Keraani \cite{K01}) was invoked to prove the existence of
a global but nonscattering solution $u_\crit$ standing exactly at
the threshold between scattering and nonscattering.  The profile
decomposition lemma is again invoked to prove that the flow of
$u_\crit$ is a precompact subset of $H^1$, which then implies that
$u_\crit$ remains spatially localized uniformly in time.  This
uniform localization enabled the use of a local virial identity to
establish, with the aid of the sharp Gagliardo-Nirenberg inequality,
a strictly positive lower bound on the convexity (in time) of the
local mass of $u_\crit$.  Mass conservation is then violated at a
sufficiently large time.

In this paper, we show that the above program carries over to the
non-radial setting with the addition of two key ingredients.  First,
in \S\ref{S:profile}, we introduce a profile decomposition lemma
that applies to non-radial $H^1$ sequences.  To compensate for the
lack of localization at the origin induced by radiality, a spatial
translation sequence is needed.  We also here adapt the proof given
in \cite{HR2} of the energy Pythagorean expansion (Lemma
\ref{C:energy_expand}) to apply to non-radial sequences; in
\cite{HR2}, an inessential application of the compact embedding
$H_\textnormal{rad}^1 \to L^4$ was used at one point.   The profile
decomposition and concentration compactness techniques are
previously used in works of Keraani \cite{K01}, Gerard \cite{G96},
see also Bahouri and Gerard \cite{BG97}-\cite{BG99}, and originate
from P.-L. Lions \cite{L85a}-\cite{L85b}.

The application of the non-radial profile decomposition to time
slices of the flow of the critical solution $u_\crit$ yields the
existence of a continuous time translation parameter $x(t)$ such
that the translated flow $u_\crit(\cdot-x(t),t)$ is precompact in
$H^1$ (Prop. \ref{P:crit_compact}). This implies the localization of
$u_\crit(\cdot,t)$ near $x(t)$ (as opposed to the radial case, in
which localization is obtained near the origin).

Obtaining suitable control on the behavior of $x(t)$ is the main new
step beyond \cite{HR2}.  This is done by following a method
introduced by Kenig-Merle \cite{KM06b} (who applied it to the
energy-critical nonlinear wave equation).  First, we argue that by
Galilean invariance, the solution $u_\crit$ must have zero momentum
(see \S \ref{S:zeromomentum}). An appropriate selection of the phase
shift is possible in our case since our solution belongs to
$L^2$.\footnote{It could not be applied in the Kenig-Merle paper
\cite{KM06a} on the energy critical NLS since the argument there
takes place in $\dot H^1$ .}  This zero-momentum solution is then
shown in \S\ref{S:x-control} to have a near-conservation of
localized center-of-mass, which provides the desired control on the
rate of divergence of $x(t)$ (specifically, $x(t)/t \to 0$ as $t\to
\infty$).

In \S\ref{S:defocusing}, we remark on the adaptation of these
techniques to the defocusing cubic NLS in 3D.

{\bf Acknowledgments.} J.H. was supported, in part, by an NSF
postdoctoral fellowship. T.D. was partially supported by the French
ANR Grant ONDNONLIN. Part of this work was done during S.R. stay in
Paris funded by the Grant ONDNONLIN. 

T.D. and S.R. would like to
thank Frank Merle and Patrick G\'erard for fruitful discussions on
the subject.  We all thank Guixiang Xu and the referee for pointing out a few misprints and a mistake in the proof of Lemma \ref{C:energy_expand}.

\bigskip

\section{Non-radial profile and energy decompositions}
\label{S:profile}

We will make use of the Strichartz norm notation used in \cite{HR2}.
We say that $(q,r)$ is $\dot H^s$ Strichartz admissible (in 3D) if
$$
\frac2q+\frac3r=\frac32-s.
$$
Let
$$
\|u\|_{S(L^2)} = \sup_{\substack{(q,r)\; L^2 \text{ admissible} \\
2\leq r \leq 6, \;  2\leq q \leq \infty}} \|u\|_{L_t^qL_x^r}.
$$
Define
$$
\|u\|_{S(\dot H^{1/2})} = \sup_{\substack{(q,r)\; \dot H^{1/2} \text{ admissible} \\
3\leq  r \leq 6^-, \; 4^+ \leq q \leq \infty}} \|u\|_{L_t^qL_x^r}\,,
$$
where $6^-$ is an arbitrarily preselected and fixed number $<6$;
similarly for $4^+$.

\begin{lemma}[Profile expansion]
 \label{L:pe}
Let $\phi_n(x)$ be a uniformly bounded sequence in $H^1$. Then for
each $M$ there exists a subsequence of $\phi_n$, also denoted
$\phi_n$, and
\begin{enumerate}
\item
for each $1\leq j \leq M$, there exists a (fixed in $n$)
profile $\psi^j(x)$ in $H^1$,
\item
for each $1\leq j\leq M$, there exists a sequence (in $n$) of time
shifts $t^j_n$,
\item
for each $1\leq j \leq M$, there exists a sequence (in $n$) of space
shifts $x_n^j$,
\item
there exists a sequence (in $n$) of remainders $W_n^M(x)$ in $H^1$,
\end{enumerate}
such that
$$
\phi_n(x) = \sum_{j=1}^M e^{-it^j_n\Delta}\psi^j(x-x_n^j) +
W^M_n(x).
$$
The time and space sequences have a pairwise divergence property,
i.e., for $1\leq j \neq k \leq M$, we have
\begin{equation}
 \label{E:pwdiverge}
\lim_{n\to +\infty} |t^j_n - t^k_n| + |x^j_n - x^k_n|= +\infty .
\end{equation}
The remainder sequence has the following asymptotic smallness
property\footnote{We can always pass to a subsequence in $n$ with
the property that $\| e^{it\Delta}W^M_n\|_{S(\dot H^{1/2})}$
converges. Therefore, we use $\lim$ and not $\limsup$ or $\liminf$.
Similar remarks apply for the limits that appear in the Pythagorean
expansion.}:
\begin{equation}
 \label{E:remainder_small}
\lim_{M\to +\infty} \left[ \lim_{n\to +\infty} \|
e^{it\Delta}W^M_n\|_{S(\dot H^{1/2})} \right] = 0 .
\end{equation}
For fixed $M$ and any $0\leq s \leq 1$, we have the asymptotic
Pythagorean expansion
\begin{equation}
\label{E:pythag} \|\phi_n\|_{\dot H^s}^2 = \sum_{j=1}^M
\|\psi^j\|_{\dot H^s}^2 + \|W^M_n\|_{\dot H^s}^2 + o_n(1).
\end{equation}
\end{lemma}

\begin{remark}
If the assumption that $\phi_n$ is uniformly bounded in $H^1$ is
weakened to the assumption that $\phi_n$ is uniformly bounded in
$\dot H^{1/2}$, then the above profile decomposition remains valid
provided a scaling parameter $\lambda$ is also involved, similar to
the theorem in \cite{K01}. However, it is not needed for the results
of this note and for simplicity of exposition the proof is omitted.
\end{remark}

\begin{proof}%
The proof is very close to the one of \cite[Lemma 5.2]{HR2}. We also
refer to \cite{K01} for a similar result in the energy-critical
case.

\medskip

\noindent\emph{Step 1. Construction of $\psi_n^1$.} Let
$A_1=\limsup_n \|e^{it\Delta}\phi_n\|_{L^{\infty}_tL^3_x}$. If
$A_1=0$, we are done. Indeed, for an arbitrary
$\dot{H}^{1/2}$-admissible couple $(q,r)$ we have
$$
\left\|e^{it\Delta}\phi_n \right\|_{L^q_tL^r_x} \leq
\left\|e^{it\Delta}\phi_n\right\|_{L^4_tL^6_x}^{\theta}\,
\left\|e^{it\Delta}\phi_n\right\|^{1-\theta}_{L^{\infty}_tL^3_x}
\quad\text{with} \quad \theta=\frac{4}{q}\in(0,1).
$$
Noting that $\left\|e^{it\Delta}\phi_n\right\|_{L^4_tL^6_x}\leq
C\|\phi_n\|_{\dot{H}^{1/2}}$, we get that $\limsup_n
\|e^{it\Delta}\phi_n\|_{S(\dot{H}^{1/2})}=0$, and we can take
$\psi^j=0$ for all $j$.

If $A_1>0$, let
$$
c_1=\limsup_n \|\phi_n\|_{H^1}.
$$
Extracting a subsequence from $\phi_n$, we show that there exist
sequences $t_n^1$, $x_n^1$ and a function $\psi^1\in H^1$ such that
\begin{gather}
 \label{E:weakCV1}
e^{it_n^1\Delta}\phi_n(\cdot+x_n^1)\rightharpoonup \psi^1 \text{ weakly in }H^1,\\
 \label{E:boundA1}
Kc_1^4\left\|\psi^1\right\|_{\dot{H}^{1/2}}\geq A_1^5,
\end{gather}
where $K>0$ is a constant independent of all parameters. %

Let $r=\frac{16c_1^2}{A_1^2}$ and $\chi_r$ be a radial Schwartz
function such that $\hat{\chi}_r (\xi)=1$ for $\frac{1}{r}\leq
|\xi|\leq r$, and $\supp \chi_r\subset
\left[\frac{1}{2r},2r\right]$. By the arguments of \cite{HR2}, there
exists sequences $t_n^1$, $x_n^1$ such that
$$
|\chi_r \ast e^{it_n^1\Delta}\phi_n(x_n^1)|\geq \frac{A_1^3}{32c_1^2}.
$$
Pass to a subsequence so that
$e^{it_n^1\Delta}\phi_n(\cdot+x_n^1)\rightharpoonup \psi^1$ weakly
in $H^1$. In \cite{HR2} the functions $\phi_n$ are radial, and thus,
by the radial Gagliardo-Nirenberg inequality, one can show that
$x_n^1$ is bounded in $n$, which is not necessarily the case here.
As in \cite{HR2}, the estimate $\|\chi_r\|_{\dot{H}^{-1/2}}\leq r$
yields, together with Plancherel and Cauchy-Schwarz inequalities,
the estimate \eqref{E:boundA1}.

Next, define $W_n^1(x)= \phi_n(x)-e^{-i t_n^1 \lap}
\psi^1(x-x_n^1)$. Since $e^{i t_n^1 \lap} \phi_n(\cdot+x_n^1)
\rightharpoonup \psi^1$ in $H^1$, expanding
$\|W_n^1\|^2_{\dot{H}^s}$ as an inner product and using the
definition of $W_n^1$, we obtain
$$
\lim_{n \to \infty} \|W_n^1\|^2_{\dot{H}^s} = \lim_{n \to \infty}
\|e^{i t_n^1 \lap} \phi_n\|^2_{\dot{H}^s} -
\|\psi^1\|^2_{\dot{H}^s}, \quad 0 \leq s \leq 1,
$$
which yields \eqref{E:pythag} for $M=1$. %

\medskip

\noindent\emph{Step 2. Construction of $\psi^j$ for $j\geq 2$.}  We construct the functions $\psi^j$ inductively, applying
Step 1 to the sequences (in $n$) $W_n^{j-1}$. Let $M\geq 2$.
Assuming that $\psi^j$, $x_n^j$, $t_n^j$ and $W_n^j$ are known for $j\in\{1,\ldots\,
M-1\}$, we consider
$$
A_{M}=\limsup_n \left\|W_n^{M-1}\right\|_{L^{\infty}_tL^3_x}.
$$
If $A_M=0$, we take, as in Step 1, $\psi^j=0$ for $j\geq M$. Assume $A_M>0$.
Applying Step 1 to the sequence %
$W_n^{M-1}$,
we obtain, extracting if necessary, %
sequences $x_n^{M}$, $t_n^{M}$ and a function $\psi^{M}\in H^1$ such
that  
\begin{gather}
 \label{E:weakCVM}
e^{it_n^{M}\Delta}W_n^{M-1}(\cdot+x_n^{M})\rightharpoonup \psi^M \text{ weakly in }H^1,\\
 \label{E:boundAM}
 Kc_M^4\left\|\psi^{M}\right\|_{\dot{H}^{1/2}}\geq A_{M}^5, 
\text{ where }c_M=\limsup_n \|W_n^{M-1}\|_{H^1}.
\end{gather}
We then define $W_n^M(x) = W_n^{M-1}(x) - e^{-i t_n^M \lap}
\psi^M(x-x_n^M)$. 

We next show \eqref{E:pwdiverge} and \eqref{E:pythag} by induction. Assume that \eqref{E:pythag} holds at rank $M-1$.
Expanding 
$$\left\|W_n^M \right\|^2_{\dot{H}^s}=\left\|e^{it_n^M\lap}W_n^M(\cdot+x_n^M) \right\|^2_{\dot{H}^s}=\left\|e^{it_n^M\lap}W_n^{M-1}(\cdot+x_n^M) - 
\psi^M\right\|^2_{\dot{H}^s}$$
and using the
weak convergence \eqref{E:weakCVM}, we obtain directly \eqref{E:pythag} at rank $M$.

Assume that the condition \eqref{E:pwdiverge} holds for $j,k\in \{1,\ldots,M-1\}$. Let $j\in
\{1,\ldots, M-1\}$. Then (here, $W_n^0=\phi_n$),
\begin{equation*}
-e^{it_n^j\Delta}W_n^{M-1}(x+x_n^j) +
e^{it_n^j\Delta}W_n^{j-1}(x+x_n^j) - \psi^j(x) = \sum_{k=j+1}^{M-1}
e^{i(t_n^j-t_n^k)\Delta}\psi^k (x+x_n^j-x_n^k).
\end{equation*}
By the orthogonality condition \eqref{E:pwdiverge}, the right hand
side converges to $0$ weakly in $H^1$ as $n$ tends to infinity.
Furthermore, by the definition of $W_n^{j-1}$,
$$
e^{it_n^j\Delta}W_n^{j-1}(x+x_n^j)-\psi^j(x)\underset{n\rightarrow +
\infty}{\rightharpoonup} 0 \text{ weakly in }H^1.
$$
Thus, $e^{it_n^j\Delta}W_n^{M-1}(x+x_n^j)$ must go to $0$ weakly in
$H^1$. From \eqref{E:weakCVM}, we deduce, if $\psi^M\neq 0$ 
$$
\lim_{n\rightarrow +\infty} |x_n^j-x_n^M|+|t_n^j-t_n^M|=+\infty,
$$
which shows that \eqref{E:pwdiverge} must also holds for $k=M$. 

It remains to show \eqref{E:remainder_small}. Note that by \eqref{E:pythag}, $c_M\leq c_1$ for all $M$. If for all $M$, $A_M>0$, we have by
\eqref{E:pythag}
\begin{equation*}
\sum_{M\geq 1} A_M^{10} \leq K^2c_1^8\sum_{n\geq 1}
\|\psi^M\|_{\dot{H}^{1/2}}^{2}\leq K^2c_1^8\limsup
\|\phi_n\|_{\dot{H}^{1/2}}^2<\infty,
\end{equation*}
which shows that $A_M$ tends to $0$ as $M$ goes to $\infty$,
yielding \eqref{E:remainder_small} and concluding the proof of Lemma
\ref{L:pe}.
\end{proof}

\begin{lemma}[Energy Pythagorean expansion]
 \label{C:energy_expand}
In the situation of Lemma \ref{L:pe}, we have
\begin{equation}
\label{E:en_expand} E[\phi_n] = \sum_{j=1}^M
E[e^{-it^j_n\Delta}\psi^j] + E[W^M_n] + o_n(1) .
\end{equation}
\end{lemma}

\begin{proof}
According to \eqref{E:pythag}, it suffices to establish for all $M\geq 1$,
\begin{equation}
 \label{E:pythagL4}
\|\phi_n\|_{L^4}^4=\sum_{j=1}^M
\big\|e^{-it^j_n\Delta}\psi^j\big\|_{L^4}^4
+\big\|W^M_n\big\|_{L^4}^4+o_n(1).
\end{equation}
\noindent\emph{Step 1. Pythagorean expansion of a sum of orthogonal
profiles.} We show that if $M\geq 1$ is fixed, orthogonality
condition  \eqref{E:pwdiverge} implies
\begin{equation}
 \label{E:pythagpsi}
\left\|\sum_{j=1}^M
e^{-it^j_n\Delta}\psi^j(\cdot-x_n^j)\right\|_{L^4}^4 =\sum_{j=1}^M
\left\|e^{-it^j_n\Delta}\psi^j\right\|_{L^4}^4+o_n(1).
\end{equation}
By reindexing, we can arrange so that there is $M_0\leq M$ such that
\begin{itemize}
\item
For $1\leq j \leq M_0$, we have that $t_n^j$ is bounded in $n$.
\item
For $M_0+1\leq j \leq M$, we have that $|t_n^j|\to \infty$ as $n\to \infty$.
\end{itemize}
By passing to a subsequence, we may assume that for each $1\leq j
\leq M_0$, $t_n^j$ converges (in $n$), and by adjusting the profiles
$\psi^j$ we can take $t_n^j=0$.

Note that
\begin{equation}
 \label{E:timedecay}
M_0 +1\leq k\leq M  \Longrightarrow \lim_{n\rightarrow+\infty}
\left\|e^{-it_n^k\Delta}\psi^k\right\|_{L^4}=0.
\end{equation}
Indeed, in this case $|t_n^k| \to \infty$ as $n\to \infty$. For a
function $\tilde \psi\in \dot H^{3/4}\cap L^{4/3}$, from Sobolev
embedding and the $L^p$ space-time decay estimate of the linear
flow, we obtain
$$
\|e^{-it_n^k\Delta}\psi^k \|_{L^4} \leq c\|\psi^k-\tilde
\psi\|_{\dot H^{3/4}} +
\frac{c}{|t_n^k|^{3/4}}\|\tilde\psi\|_{L^{4/3}}.
$$
By approximating $\psi^k$ by $\tilde \psi\in C_c^\infty$ in $\dot
H^{3/4}$ and sending $n\to \infty$, we obtain \eqref{E:timedecay}.

By \eqref{E:pwdiverge}, if $1\leq j<k\leq M_0$, $\lim_n
|x_n^j-x_n^k|=+\infty$, and thus, %
it implies
$$
\left\|\sum_{j=1}^{M_0} \psi^j \left(\cdot-x_n^j\right)
\right\|_{L^4}^4 = \sum_{j=1}^{M_0}
\big\|\psi^j\big\|_{L^4}^4+o_n(1),
$$
which yields, together with \eqref{E:timedecay}, expansion \eqref{E:pythagpsi}.

\medskip

\noindent\emph{Step 2. End of the Proof.}
We first note
\begin{equation}
 \label{E:remaindersmallL4}
\lim_{M_1\rightarrow+\infty} \left(\lim_{n\rightarrow+\infty}
\left\|W_n^{M_1}\right\|_{L^4}\right) = 0.
\end{equation}
Indeed,
\begin{align*}
\|W_n^{M_1}\|_{L_x^4} &\leq \|e^{it\Delta}W_n^{M_1}\|_{L_t^\infty L_x^4} \\
&\leq \|e^{it\Delta}W_n^{M_1}\|_{L_t^\infty L_x^3}^{1/2} \|e^{it\Delta}W_n^{M_1}\|_{L_t^\infty \dot H_x^1}^{1/2} \\
&\leq \|e^{it\Delta}W_n^{M_1}\|_{L_t^\infty L_x^3}^{1/2} \sup_n \|\phi_n\|_{H^1}^{1/2}.
\end{align*}
By \eqref{E:remainder_small}, we get \eqref{E:remaindersmallL4}.

Let $M\geq 1$ and $\eps>0$. Note that $\{\phi_n\}_n$ is uniformly
bounded in $L^4$, since it is uniformly bounded in $H^1$ by the
hypothesis; furthermore, by \eqref{E:remaindersmallL4} $\{W_n^M\}_n$
is also uniformly bounded in $L^4$. Thus,
we can choose $M_1\geq M$ and $N_1$ such that for $n\geq N_1$, we
have
\begin{multline}
\label{E:L4a}
\Big|\| \phi_n-W_n^{M_1}\|_{L^4}^4 -\|\phi_n\|_{L^4}^4\Big|
+\Big|\|W_n^M-W_n^{M_1}\|_{L^4}^4-\|W_n^M\|_{L^4}^4\Big|\\
\leq C\left[\Big(\sup_n\|\phi_n\|_{L^4}^3+\sup_n\|W_n^M\|_{L^4}^3\Big)\|W_n^{M_1}\|_{L^4}+\|W_n^{M_1}\|_{L^4}^4\right]\leq  \eps.
\end{multline}
By \eqref{E:pythagpsi} , we get $N_2\geq N_1$ such that for $n\geq  N_2$,
\begin{equation}
\label{E:L4b}
\bigg|\left\|\phi_n-W_n^{M_1}\right\|_{L^4}^4 -\sum_{j=1}^{M_1} \left\|e^{-it^j_n\Delta}\psi^j\right\|_{L^4}^4\bigg|\leq \eps.
\end{equation}
Using the definition of $W_n^j$, expand $W_n^M-W_n^{M_1}$ to obtain
$$W_n^M-W_n^{M_1}=\sum_{j=M+1}^{M_1} e^{-it^j_n\Delta}\psi^j(\cdot-x_j).$$ 
By \eqref{E:pythagpsi} there exists $N_3\geq N_2$ such that for $n\geq  N_2$,
\begin{equation}
\label{E:L4c}
\bigg|\left\|W_n^M-W_n^{M_1}\right\|_{L^4}^4 -\sum_{j=M+1}^{M_1} \left\|e^{-it^j_n\Delta}\psi^j\right\|_{L^4}^4\bigg|\leq \eps.
\end{equation}
By \eqref{E:L4a}, \eqref{E:L4b} and \eqref{E:L4c}, we obtain that for $n\geq N_3$,
$$ \bigg|\left\|\phi_n\right\|_{L^4}^4 -\sum_{j=1}^{M} \left\|e^{-it^j_n\Delta}\psi^j\right\|_{L^4}^4-\left\|W_n^M\right\|_{L^4}^4\bigg|\leq 4\eps,$$
which concludes the proof of \eqref{E:pythagL4}.
\end{proof}

\section{Outline of the proof of the main result}
\label{S:outline}
Let $u(t)$ be the corresponding $H^1$ solution to
\eqref{E:NLS}-\eqref{E:initialdata}. By Theorem 1.1(1)(a) in
\cite{HR2} the solution is globally well-posed, so our goal is to
show that
\begin{equation}
 \label{E:H12scatter}
\Vert u \Vert_{S(\dot{H}^{1/2})} < \infty.
\end{equation}
This combined with Proposition 2.2 from \cite{HR2} will give $H^1$
scattering. We will use the strategy of \cite{KM06a}. We shall say
that $\textnormal{SC}(u_0)$ holds if \eqref{E:H12scatter} is true
for the solution $u(t)$ generated from $u_0$.

By the small data theory there exists $\delta > 0$ such that if
$M[u]E[u]< \delta$ and $\|u_0\|_{L^2}\|\nabla
u_0\|_{L^2}<\|Q\|_{L^2}\|\nabla Q\|_{L^2}$, then
\eqref{E:H12scatter} holds. For each $\de > 0$ define the set $S_\de$
to be the collection of all such initial data in $H^1$:
$$
S_\de = \{u_0 \in H^1 \quad \text{with} \quad M[u]E[u]< \delta \quad
\text{and} \quad \|u_0\|_{L^2} \|\nabla u_0\|_{L^2} < \|Q\|_{L^2}
\|\nabla Q\|_{L^2}\}.
$$
Next define $(ME)_\crit = \sup \{\de : u_0 \in S_\de \Longrightarrow
SC(u_0) ~~\text{holds} \}$. If $(ME)_\crit = M[Q]E[Q]$, then we are
done, so we assume
\begin{equation}
 \label{assumption<}
(ME)_\crit < M[Q]E[Q].
\end{equation}
Then there exists a sequence of solutions $u_n$ to \eqref{E:NLS}
with $H^1$ initial data $u_{n,0}$ (rescale all of them to have
$\Vert u_n \Vert_{L^2} = 1$ for all $n$) such that $\|\nabla
u_{n,0}\|_{L^2} < \|Q\|_{L^2}\|\nabla Q\|_{L^2}$ and $E[u_n]
\searrow (ME)_\crit$ as $n\to +\infty$, for which
$\text{SC}(u_{n,0})$ does not hold for any $n$.

The next proposition gives the existence of an $H^1$ solution
$u_\text{c}$ to \eqref{E:NLS} with initial data $u_{\text c,0}$ such
that $\|u_\crin\|_{L^2} \|\nabla u_\crin\|_{L^2}<
\|Q\|_{L^2}\|\nabla Q\|_{L^2}$ and $M[u_\crit ]E[u_\crit ]=
(ME)_\text{c}$ for which $\text{SC}(u_\crin)$ does not hold. This
will imply that $K = \{ \, u_\crit(\cdot-x(t), t) \, | \, 0\leq t <
+\infty \, \}$ is precompact in $H^1$ (Proposition
\ref{P:crit_compact}). As a consequence (see Corollary
\eqref{C:unif_small}) we obtain that for each $\epsilon>0$, there is
an $R> 0$ such that, uniformly in $t$, we have
\begin{equation}
\label{E:localization}
\int_{|x+x(t)|>R} |\nabla u_\crit(t,x)|^2 dx \leq \epsilon.
\end{equation}
This together with the hypothesis of zero momentum (which can always
be achieved by Galilean invariance -- see \S \ref{S:zeromomentum})
provides a control on the growth of $x(t)$ (Lemma
\ref{L:x-control}).  Finally, the rigidity theorem  (Theorem
\ref{T:rigidity}), which appeals to this control on $x(t)$ and the
uniform localization \eqref{E:localization}, will lead to a
contradiction that such critical element exists (unless it is
identically zero) which will conclude the proof.

\begin{proposition}[Existence of a critical solution]
 \label{P:exist_crit}
Assume \eqref{assumption<}. Then there exists a global
$(T^*=+\infty)$ solution $u_\crit$ in $H^1$ with initial data
$u_\crin$ such that $\|u_\crin\|_{L^2}=1$,
$$
E[u_\crit] = (ME)_\crit < M[Q]E[Q],
$$
$$
\|\nabla u_\crit(t)\|_{L^2} < \|Q\|_{L^2}\|\nabla Q\|_{L^2} \quad
\text{for all } 0\leq t<+\infty,
$$
and
$$
\|u_\crit \|_{S(\dot H^{1/2})} = +\infty.
$$
\end{proposition}
\begin{proof}
The proof closely follows the proof of \cite[Prop 5.4]{HR2}.
\end{proof}

\begin{proposition}[Precompactness of the flow of the critical solution]
 \label{P:crit_compact}
With $u_\crit$ as in Proposition \ref{P:exist_crit}, there exists a
continuous path $x(t)$ in $\mathbb{R}^3$ such that
$$
K = \{ \, u_\crit(\cdot -x(t),t) \, | \, t\in [0,+\infty) \, \} \subset
H^1
$$
is precompact in $H^1$ (i.e., $\bar{K}$ is compact).
\end{proposition}
\begin{proof}
For convenience, we write $u=u_\crit$. We argue by contradiction. By
the arguments in Appendix \ref{A:lifting}, we can assume that there
exists $\eta>0$ and a sequence $t_n$ such that for all $n \neq n'$,
\begin{equation}
\label{E:negation} \inf_{x_0\in \mathbb{R}^3}
\|u(\cdot-x_0,t_n)-u(\cdot,t_{n'})\|_{H^1} \geq \eta.
\end{equation}
Take $\phi_n=u(t_n)$ in the profile expansion lemma (Lemma
\ref{L:pe}).  The remainder of the argument closely follows the
proof of \cite[Prop 5.5]{HR2}.
\end{proof}

\begin{corollary}[Precompactness of the flow implies uniform localization]
 \label{C:unif_small}
Let $u$ be a solution to \eqref{E:NLS} such that
$$
K = \{ \, u(\cdot-x(t), t) \, | \, t\in [0,+\infty) \, \}
$$
is precompact in $H^1$.  Then for each $\epsilon>0$, there exists
$R> 0$ so that
$$
\int_{|x+x(t)|> R} |\nabla u(x,t)|^2 + |u(x,t)|^2 +|u(x,t)|^4\, dx \leq
\epsilon, \quad \text{for all }0\leq t <+\infty.
$$
\end{corollary}

\begin{proof}
If not, then there exists $\epsilon>0$ and a sequence of times $t_n$
such that
$$
\int_{|x+x(t_n)|>n} |\nabla u(x,t_n)|^2 + |u(x,t_n)|^2 +|u(x,t_n)|^4\, dx \geq
\epsilon,
$$
or, by changing variables,
\begin{equation}
\label{int>epsilon}
\int_{|x|>n} |\nabla u(x - x(t_n),t_n)|^2 + |u(x - x(t_n),t_n)|^2+ |u(x - x(t_n),t_n)|^4\,
dx \geq \epsilon.
\end{equation}
Since $K$ is precompact, there exists $\phi \in H^1$ such that,
passing to a subsequence of $t_n$, we have  $u(\cdot - x(t_n),t_n)
\to \phi$ in $H^1$.  By \eqref{int>epsilon}
$$
\forall R>0,\quad \int_{|x|>R} |\nabla \phi(x)|^2 + |\phi(x)|^2 +|\phi(x)|^4\geq \epsilon,
$$
which is a contradiction with the fact that $\phi\in H^1$.
\end{proof}

\section{Zero momentum of the critical solution}
\label{S:zeromomentum}
\begin{proposition}
 \label{P:zeromomentum}
Assume \eqref{assumption<} and let $u_\crit$ be the critical
solution constructed in Section \ref{S:outline}. Then its conserved
momentum $P[u_\crit]= \I \int \bar u_\crit \nabla u_\crit \, dx$ is
zero.
\end{proposition}
\begin{proof}
Consider for some $\xi_0\in \mathbb{R}^3$ the transformed
solution
$$
w_c(x,t) = e^{ix\cdot\xi_0} e^{-it|\xi_0|^2} u_{\crit}(x-2\xi_0t,t).
$$
We compute
$$
\|\nabla w_c\|_{L^2}^2 = |\xi_0|^2M[u_{\crit}] + 2\xi_0\cdot
P[u_{\crit}] + \|\nabla u_{\crit}\|_{L^2}^2.
$$
Observe that $M[w_c]=M[u_{\crit}]$ and
$$
E[w_c] = \frac12|\xi_0|^2 M[u_{\crit}] + \xi_0 \cdot P[u_{\crit}] + E[u_{\crit}].
$$
To minimize $E[w_c]$, we take $\xi_0 = -P[u_{\crit}]/M[u_{\crit}]$.

Assume $P[u_c] \neq 0$. Choose $\xi_0 = -\frac{P[u_c]}{M[u_c]}$.
Then $P[w_c] = 0$ and
\begin{equation}
 \label{E:zeromomentumprop}
M[w_c]=M[u_c], ~~ E[w_c] = E[u_c] - \frac12 \,
\frac{P[u_c]^2}{M[u_c]}, ~~ \|\nabla w_c\|^2_{L^2} = \|\nabla u_c
\|^2_{L^2} - \frac{P[u_c]^2}{M[u_c]}.
\end{equation}
Thus, $M[w_c]E[w_c]<M[u_c]E[u_c]$, $\|w_c\|_{L^2}\|\nabla
w_c\|_{L^2}<\|Q\|_{L^2}\|\nabla Q\|_{L^2}$. By Proposition
\ref{P:exist_crit}, $\|u_c\|_{S(\dot{H}^{1/2})}=+\infty$, and hence,
$\|w_c\|_{S(\dot{H}^{1/2})}=+\infty$, which contradicts the
definition of $u_c$.
\end{proof}

\section{Control of the spatial translation parameter}
\label{S:x-control}

Observe that
\begin{equation}
 \label{E:Dmomentum}
\frac{\partial}{\partial t} \int x |u(x, t)|^2 \, dx = 2 \im \int
\bar{u} \, \nabla u \, dx  = 2 P[u].
\end{equation}

Since $P[u_\crit]=0$ (see Prop. \ref{P:zeromomentum}), it follows
that $\ds \int x |u_\crit(x, t)|^2 \, dx = \text{const}$, provided
it is finite.  We will replace this identity with a version
localized to a suitably large radius $R>0$.  Provided the
localization $R$ is taken large enough over an interval $[t_0,t_1]$
to envelope the entire path $x(t)$ over $[t_0,t_1]$, we can exploit
the localization of $u_\crit$ in $H^1$ around $x(t)$ (induced by the
precompactness of the translated flow $u_\crit(\cdot-x(t),t)$) and
the zero-momentum property to prove that the localized center of
mass is \emph{nearly} conserved.  The parameter $x(t)$ is then
constrained from diverging too quickly to $+\infty$ by the
localization of $u_\crit$ in $H^1$ around $x(t)$ and the near
conservation of localized center of mass. We refer to \cite[Lemma
5.5]{KM06b} for a similar proof in the case of the energy-critical
non-radial wave equation.

\begin{lemma}
 \label{L:x-control}%
Let $u$ be a solution of \eqref{E:NLS} defined on $[0,+\infty)$ such
that $P[u]= 0$ and $K = \{ u(\cdot - x(t),t) | \, t \in [0,\infty)
\}$ is precompact in $H^1$, for some continuous function $x(\cdot)$.
Then
\begin{equation}
 \label{E:x-control}
\frac{x(t)}{t} \to 0 \quad \text{as} ~~ t \to +\infty.
\end{equation}
\end{lemma}

\begin{proof}
Assume that \eqref{E:x-control} does not hold. Then there exists a
sequence $t_n \to +\infty$ such that $|x(t_n)|/t_n \geq \epsilon_0$
for some $\epsilon_0 > 0 $.  Without loss of generality we may assume
$x(0)=0$.  For $R>0$, let
$$t_0(R) = \inf \{t \geq 0: |x(t)| \geq R \} \,,$$
i.e., $t_0(R)$ is the first time when $x(t)$ reaches the boundary of
the ball of radius $R$.  By continuity of $x(t)$, the value $t_0(R)$
is well-defined.  Moreover, the following properties hold: (1)
$t_0(R) > 0$; (2) $|x(t)| < R$ for $0 \leq t < t_0(R)$; and (3)
$|x(t_0(R))| = R$.

Define $R_n = |x(t_n)|$ and $\tilde t_n = t_0(R_n)$. Note that $t_n
\geq \tilde t_n$, which combined with $|x(t_n)|/t_n\geq \epsilon_0$
gives $R_n/ \tilde t_n \geq \epsilon_0$.  Since $t_n \to +\infty$
and $|x(t_n)|/t_n \geq \epsilon_0$, we have $R_n=|x(t_n)| \to
+\infty$.  Thus, $\tilde t_n = t_0(R_n)\to +\infty$.  At this point,
we can forget about $t_n$; we will work on the time interval
$[0,\tilde t_n]$ and the only data that we will use in the remainder
of the proof is:
\begin{enumerate}
\item for $0\leq t< \tilde t_n$, we have $|x(t)|<R_n$;
\item $|x(\tilde t_n)| = R_n$;
\item $\dfrac{R_n}{\tilde t_n} \geq \epsilon_0$ and $\tilde t_n \to +\infty$.
\end{enumerate}

By the precompactness of $K$ and Corollary \ref{C:unif_small}, it follows that for any $\ep > 0$ there
exists $R_0(\ep)\geq 0$ such that for any $t \geq 0$,
\begin{equation}
 \label{E:compact1}
\int\limits_{|x+x(t)| \geq R_0(\ep)} \left( |u|^2 + |\nabla u|^2
\right) \, dx \leq \ep.
\end{equation}
We will select $\epsilon>0$ appropriately later.

For $x\in \mathbb{R}$, let $\theta(x)\in C_c^\infty(\mathbb{R})$ be
such that $\theta(x) = x$, for $-1\leq x \leq 1$, $\theta(x)=0$ for
$|x|\geq 2^{1/3}$, $|\theta(x)|\leq |x|$, $\|\theta'\|_\infty \leq
4$, and $\|\theta\|_\infty\leq 2$.  For $x=(x_1,x_2,x_3)\in
\mathbb{R}^3$, let $\phi(x) =
(\theta(x_1),\theta(x_2),\theta(x_3))$.  Then $\phi(x)=x$ for
$|x|\leq 1$ and $\|\phi\|_{\infty}\leq 2$.
For $R>0$, set $\phi_R(x) = R\phi (x/R)$. Let $z_R: \mathbb{R}\to
\mathbb{R}^3$ be the truncated center of mass given by
$$
z_R(t) = \int \phi_R(x) \, |u(x,t)|^2 \, dx \, .
$$
Then $z'_R(t) = ([z'_R(t)]_1, [z'_R(t)]_2, [z'_R(t)]_3)$, where
$$
[z'_R(t)]_j = 2 \im \int \theta'(x_j/R)\, \partial_j u \; \bar{u}
\, dx.
$$
Note that $\theta'(x_j/R)=1$ for $|x_j|\leq 1$.  By the zero momentum property,
$$
\im \int_{|x_j|\leq R} \partial_ju \, \bar{u} = - \im \int_{|x_j|> R} \partial_ju \, \bar{u},
$$
and thus,
$$
[z_R'(t)]_j = -2 \im \int_{|x_j|\geq R} \, \partial_j u \; \bar{u}
\, dx + 2 \im \int_{|x_j|\geq R} \theta'(x_j/R)\partial_j u \;
\bar{u} \, dx \,,
$$
from which we obtain by Cauchy-Schwarz,
\begin{equation}
\label{E:z-prime-bound}
|z_R'(t)| \leq 5 \int_{|x|\geq R} (|\nabla u|^2 + |u|^2) \,.
\end{equation}

Set $\tilde R_n = R_n + R_0(\ep)$. Note that for $0 \leq t \leq
 \tilde t_n$ and $|x|>\tilde R_n$, we have $|x+x(t)| \geq \tilde R_n -
R_n = R_0(\ep)$, and thus, \eqref{E:z-prime-bound} and
\eqref{E:compact1} give
\begin{equation}
 \label{E:derivative-z}
|z'_{\tilde R_n}(t)| \leq 5 \, \ep.
\end{equation}

Now we obtain an upper bound for $z_{\tilde R_n}(0)$ and a lower bound for $z_{\tilde R_n}(t)$.
$$z_{\tilde R_n}(0) = \int_{|x|< R_0(\ep)} \phi_{\tilde R_n}(x) \,
|u_0(x)|^2 \, dx + \int_{|x+x(0)| \geq R_0(\ep)} \phi_{\tilde
R_n}(x) |u_0(x)|^2 \, dx \, ,
$$
and hence, by \eqref{E:compact1}, we have
\begin{equation}
\label{E:z0}
|z_{\tilde R_n}(0)| \leq R_0(\ep) M[u] + 2\tilde R_n \, \ep.
\end{equation}
For $0 \leq t \leq \tilde t_n$, we split $z_{\tilde R_n}(t)$ as
\begin{align*}
z_{\tilde R_n}(t) &= \int_{|x+x(t)| \geq R_0(\ep)} \phi_{\tilde
R_n}(x) \, |u(x,t)|^2 \, dx + \int_{|x+x(t)| \leq R_0(\ep)}
\phi_{\tilde R_n} (x) \, |u(x,t)|^2 \, dx. \\
&= \text{I} + \text{II}
\end{align*}
To estimate I, we note that $|\phi_{\tilde R_n} (x)| \leq 2\tilde
R_n$ and use \eqref{E:compact1} to obtain $|\text{I}| \leq 2\tilde
R_n\epsilon$. For II, we first note that $|x| \leq |x+x(t)| + |x(t)|
\leq R_0(\epsilon) + R_n = \tilde R_n$, and thus $\phi_{\tilde
R_n}(x) = x$.  We now rewrite II as
\begin{align*}
\text{II} &=
\int_{|x+x(t)| \leq R_0(\ep)} (x+x(t)) \, |u(x,t)|^2 \, dx -
x(t) \int_{|x+x(t)| \leq R_0(\ep)} |u(x,t)|^2 \, dx\\
&= \int_{|x+x(t)| \leq R_0(\ep)} (x+x(t)) \, |u(x,t)|^2 \, dx - x(t)M[u] + x(t)\int_{|x+x(t)| \geq R_0(\ep)} \, |u(x,t)|^2 \, dx\\
&= \text{IIA} + \text{IIB} + \text{IIC}
\end{align*}
Trivially, $|\text{IIA}| \leq R_0(\epsilon)M[u]$, and by
\eqref{E:compact1}, $|\text{IIC}| \leq |x(t)|\epsilon \leq \tilde
R_n\epsilon$. Thus,
\begin{align*}
|z_{\tilde R_n}(t)| &\geq |\text{IIB}| - |\text{I}| - |\text{IIA}| - |\text{IIC}|\\
&\geq |x(t)|M[u]-R_0(\epsilon)M[u]-3\tilde R_n\epsilon \, .%
\end{align*}
Taking $t=\tilde t_n$, we get
\begin{equation}
\label{E:zt}
|z_{\tilde R_n}(\tilde t_n)| \geq \tilde R_n(M[u]-3\epsilon) - R_0(\epsilon)M[u] \, .
\end{equation}
Combining \eqref{E:derivative-z}, \eqref{E:z0}, and \eqref{E:zt}, we have
\begin{align*}
5\,\epsilon\, \tilde t_n & \geq \int_0^{\tilde t_n} |z'_{\tilde R_n}(t)| \, dt \geq \left|  \int_0^{\tilde t_n} z'_{\tilde R_n}(t) \, dt \right| \geq |z_{\tilde{R}_n}(\tilde t_n) - z_{\tilde{R}_n}(0)| \\
&\geq \tilde R_n(M[u]-5\epsilon) - 2R_0(\epsilon)M[u] \, .%
\end{align*}
Dividing by $\tilde t_n$ and using that $\tilde R_n \geq R_n$ (assume $\epsilon\leq \frac15M[u]$), we obtain
$$
5\, \epsilon \geq \frac{R_n}{\tilde t_n}(M[u]-5\epsilon) -
\frac{2R_0(\epsilon)M[u]}{\tilde t_n} \, .
$$
Since $R_n/\tilde t_n \geq \epsilon_0$, we have
$$
5\, \epsilon \geq \epsilon_0(M[u]-5\epsilon) -
\frac{2R_0(\epsilon)M[u]}{\tilde t_n} \, .
$$
Take $\epsilon=M[u]\ep_0/16$ (assume $\epsilon_0\leq 1$), and then
send $n\to +\infty$.  Since $\tilde t_n\to+\infty$, we get a
contradiction.
\end{proof}

\section{Rigidity theorem}
\label{S:rigidity}

\newcommand{\GN}{\textnormal{GN}}

We now prove the following rigidity, or Liouville-type, theorem.

\begin{theorem}[Rigidity]
 \label{T:rigidity}
Suppose $u_0 \in H^1$ satisfies $P[u_0]=0$,
\begin{equation}
 \label{E:comp1}
M[u_0]E[u_0] < M[Q]E[Q]
\end{equation}
and
\begin{equation}
 \label{E:comp2}
\|u_0\|_{L^2} \|\nabla u_0\|_{L^2} < \|Q\|_{L^2} \|\nabla Q
\|_{L^2}.
\end{equation}
Let $u$ be the global $H^1$ solution of \eqref{E:NLS} with initial
data $u_0$ and suppose that
$$
K = \{ \, u(\cdot-x(t), t) \, | \, t \in [0,+\infty) \, \} \quad
\text{is precompact in} ~ H^1.
$$
Then $u_0 = 0$.
\end{theorem}

Before beginning the proof, we recall in Lemma \ref{L:coercive}
below a few basic facts proved in \cite{HR2}.  These facts are
consequences of the Gagliardo-Nirenberg inequality
$$
\|u\|_{L^4}^4 \leq c_{\GN}\|u\|_{L^2}\|\nabla u\|_{L^2}^3
$$
with the sharp value of $c_{\GN}$ expressed as
$$
c_{\GN} = \frac{4}{3\|Q\|_2\|\nabla Q\|_{2}} \,.
$$
One also uses the relation
$$
M[Q]E[Q] = \frac16\|Q\|_{L^2}^2\|\nabla Q\|_{L^2}^2 \,,
$$
which is a consequence of the Pohozhaev identities.

\begin{lemma}
 \label{L:coercive}
If $M[u]E[u] < M[Q]E[Q]$ and $\|u_0\|_{L^2}\|\nabla u_0\|_{L^2} <
\|Q\|_{L^2}\|\nabla Q\|_{L^2}$, then for all $t$,
\begin{equation}
\label{E:coercive}
\|u(t)\|_{L^2}\|\nabla u(t)\|_{L^2} \leq \omega \|Q\|_{L^2}\|\nabla Q\|_{L^2}
\end{equation}
where $\omega=\left(\frac{M[u]E[u]}{M[Q]E[Q]}\right)^{1/2}$.
We also have the bound, for all $t$
\begin{equation}
\label{E:virial-bound} 8\|\nabla u(t)\|^2-6\|u(t)\|_{L^4}^4 \geq
8(1-\omega)\|\nabla u(t)\|_{L^2}^2 \geq 16(1-\omega)E[u].
\end{equation}
We remark that under the hypotheses here, $E[u]>0$ unless $u\equiv
0$. In fact, one has the bound $E[u] \geq \frac16\,\|\nabla
u_0\|_{L^2}^2$.
\end{lemma}

\begin{proof}[Proof of Theorem \ref{T:rigidity}]
In the proof below, all instances of a constant $c$ refer to some
absolute constant.
Let $\varphi \in C_0^{\infty}$ be radial with
$$
\varphi(x)=\left\{
\begin{array}{lll}
|x|^2 &\text{for} &|x|\leq 1\\
0 &\text{for}&|x|\geq 2
\end{array}
\right. \, .
$$
For $R>0$, define
$$z_R(t) = \int R^2 \varphi\left(\frac{x}{R}\right)
\, |u(x,t)|^2 \, dx \,.$$
Then, by direct calculation,
$$z_R'(t) = 2\I \int R\nabla\varphi\left(\frac{x}{R}\right)\cdot \nabla u(t) \; \bar u(t) \, dx$$
By the H\"older inequality,
\begin{equation}
\label{E:z'R-bound} |z_R'(t)| \leq cR\int_{|x|\leq 2R} |\nabla
u(t)||u(t)|\, dx \leq cR\|\nabla u(t)\|_{L^2}\|u(t)\|_{L^2}
\end{equation}
Also by direct calculation, we have the local virial identity
$$
z_R''(t)=4\sum_{j,k}
\int\frac{\partial^2\varphi}{\partial{x_j}\partial{x_k}}
\left(\frac{x}{R}\right) \frac{\partial u}{\partial{x_j}}
\frac{\partial \bar{u}}{\partial x_k}-\frac{1}{R^2}\int (\Delta^2
\varphi) \left(\frac{x}{R}\right)|u|^2-\int (\Delta \varphi)
\left(\frac xR\right)|u|^4.
$$
Since $\varphi$ is radial, we have
\begin{equation}
\label{E:z''R}
z''_R(t) = \left(8\int |\nabla u|^2-6\int |u|^4\right)+A_R(u(t)),
\end{equation}
where
$$A_R(u(t)) =
\begin{aligned}[t]
& 4 \sum_{j} \int \left((\partial_{x_j}^2 \varphi)\left(\frac
xR\right) - 2 \right)\left|\partial_{x_j}u\right|^2 + 4 \sum_{j \neq
k} \int\limits_{R \leq |x| \leq 2R}
\frac{\partial^2\varphi}{\partial{x_j}
\partial{x_k}} \left(\frac{x}{R}\right) \frac{\partial
u}{\partial{x_j}} \frac{\partial \bar{u}}{\partial x_k}\\
& -\frac{1}{R^2} \int (\Delta^2\varphi)\left(\frac{x}{R}\right)|u|^2
-\int \left((\Delta \varphi) \left(\frac xR\right)-6\right)|u|^4.
\end{aligned}
$$
From this expression, we obtain the bound
\begin{equation}
 \label{E:AR-bound}
|A_R(u(t))| \leq c \int_{|x|\geq R} \left( |\nabla u(t)|^2 +
\frac1{R^2}|u(t)|^2 + |u(t)|^4 \right)\, dx.
\end{equation}

We want to examine $z_R(t)$, for $R$ chosen suitably large, over a
suitably chosen time interval $[t_0,t_1]$, where $1 \ll t_0 \ll t_1
<\infty$.   By \eqref{E:z''R} and \eqref{E:virial-bound}, we have
\begin{equation}
 \label{E:rig1}
|z_R''(t)| \geq 16(1-\omega)E[u] - |A_R(u(t))|.
\end{equation}
Set $\epsilon = \frac{1-\omega}{c}E[u]$ in Corollary \ref{C:unif_small}
to obtain $R_0\geq 0$ such that $\forall \; t$,
\begin{equation}
\label{E:rig2}
\int_{|x+x(t)|\geq R_0} (|\nabla u|^2 +|u|^2+|u|^4) \leq \frac{(1-\omega)}{c}E[u].
\end{equation}
If we select $R\geq R_0+\sup_{t_0\leq t\leq t_1} |x(t)|$, then
\eqref{E:rig1} combined with the bounds \eqref{E:AR-bound} and
\eqref{E:rig2} will imply that, for all $t_0\leq t \leq t_1$,
\begin{equation}
 \label{E:rig3}
|z_R''(t)|\geq 8(1-\omega)E[u].
\end{equation}

By Lemma \ref{L:x-control}, there exists $t_0\geq 0$ such that for all $t\geq t_0$, we have $|x(t)| \leq \eta t$, with $\eta>0$ to be selected later.  Thus, by taking $R=R_0+\eta t_1$, we obtain that \eqref{E:rig3} holds for all $t_0\leq t \leq t_1$.  Integrating \eqref{E:rig3} over $[t_0,t_1]$, we obtain
\begin{equation}
 \label{E:rig4}
|z_R'(t_1)-z_R'(t_0)| \geq 8(1-\omega)E[u](t_1-t_0).
\end{equation}
On the other hand, for all $t_0\leq t\leq t_1$, by
\eqref{E:z'R-bound} and \eqref{E:coercive}, we have
\begin{equation}
 \label{E:rig5}
\begin{aligned}
|z'_R(t)| & \leq cR \|u(t)\|_{L^2}\|\nabla u(t)\|_{L^2} \leq cR  \|Q\|_{L^2}\|\nabla Q\|_{L^2} \\
&\leq c  \|Q\|_{L^2}\|\nabla Q\|_{L^2}(R_0+\eta t_1).
\end{aligned}
\end{equation}
Combining \eqref{E:rig4} and \eqref{E:rig5}, we obtain
$$
8(1-\omega)E[u](t_1-t_0) \leq 2c\|Q\|_{L^2}\|\nabla
Q\|_{L^2}(R_0+\eta t_1).
$$
Recall that $\omega$ and $R_0$ are constants depending only upon
$(M[u]E[u])/(M[Q]E[Q])$, while $\eta>0$ is yet to be specified and
$t_0=t_0(\eta)$.  Put $\eta=(1-\omega)E[u]/(c\|Q\|_2\|\nabla Q\|_2)$
and then send $t_1\to +\infty$ to obtain a contradiction unless
$E[u]=0$ which implies $u\equiv 0$.
\end{proof}

To complete the proof of Theorem \ref{T:non-radial}, we just apply
Theorem \ref{T:rigidity} to $u_\crit$ constructed in Proposition
\ref{P:exist_crit}, which by Propositions \ref{P:crit_compact} and
\ref{P:zeromomentum}, meets the hypotheses in Theorem
\ref{T:rigidity}.  Thus $u_{\crit,0}=0$, which contradicts the fact
that $\|u_\crit\|_{S(\dot H^{1/2})}=\infty$. We have thus obtained
that if $\|u_0\|_{L^2}\|\nabla u_0\|_{L^2}<\|Q\|_{L^2}\|\nabla
Q\|_{L^2}$ and $M[u]E[u]<M[Q]E[Q]$, then $\text{SC}(u_0)$ holds,
i.e.\ $\|u\|_{S(\dot H^{1/2})} <\infty$.  By Proposition 2.2
\cite{HR2}, $H^1$ scattering holds.

\section{Remarks on the defocusing equation}
\label{S:defocusing}%
One may use the above arguments to show $H^1$-scattering of
solutions of the defocusing equation
\begin{gather}
 \label{E:NLS2}
i\partial_t u +\Delta u - |u|^2 \, u=0, \quad (x,t)\in
\mathbb{R}^3\times \mathbb{R},\\
 \label{E:initialdata2}
u(x,0) = u_0 \in H^1(\cR^3).
\end{gather}
In this case, scattering is already known, as a consequence of
Morawetz \cite{GV85}, or interaction Morawetz \cite{CKSTT04}
inequalities.

We argue by contradiction. If scattering does not hold, there exists
a critical solution $u_c$, which does not scatter, and such that
$M[u_c]E[u_c]$ is minimal for non-scattering solutions of
\eqref{E:NLS2}. As before, one shows that $P[u_c]=0$, and that there
exists $x(t)$ such that the set $K=\{u_c(t,\cdot-x(t)),\;t\in
\mathbb{R}\}$ is precompact in $H^1$. Note that because of the
defocusing sign of the non-linearity, we do not need to assume
$M[u_c]E[u_c]<M[Q]E[Q]$ and $\|u_c(0)\|_{L^2}\, \| \nabla
u_c(0)\|_{L^2} < \|Q\|_{L^2} \, \|\nabla Q\|_{L^2}$. The control of
the spatial translation $x(t)$ works as in Section
\ref{S:x-control}, and one concludes as in Section \ref{S:rigidity},
by a localized virial argument, using that in the defocusing case,
the second derivative of the localized variance $z_R(t)$ is
$$
z''_R(t) = \left( 8 \int |\nabla u|^2+6\int
|u|^4\right)+B_R(u(t)),$$ where $B_R$ satisfies the bound
$$
|B_R(u(t))| \leq c \int_{|x|\geq R} \left( |\nabla u(t)|^2 +
\frac1{R^2}|u(t)|^2 + |u(t)|^4 \right)\, dx.
$$
Note that the use of the virial identity is potentially more robust
since one might be able to handle variants of the NLS equation (for
example with a linear potential) that might be out of reach for
Morawetz based proofs.

\appendix

\section{A lifting lemma}
\label{A:lifting}

In this appendix, we discuss some basic analysis facts needed in the
very beginning of the proof of Prop. \ref{P:crit_compact}.

Let $G \cong \mathbb{R}^3$ act on $H^1$ by translation, i.e.,
$(x_0\cdot \phi)(x) = \phi(x-x_0)$. Write $G\backslash H^1$ for the
quotient space endowed with the quotient topology.   We represent
elements of $G\backslash H^1$ (the equivalence classes) by $[\phi]$,
and let $\pi:H^1\to G\backslash H^1$ be the natural projection.

\begin{lemma}
$G\backslash H^1$ is metrizable with metric
$$d([\phi],[\psi]) = \inf_{x_0\in \mathbb{R}^3} \|\phi(\cdot-x_0)-\psi\|_{H^1}.$$
With respect to this metric, $G\backslash H^1$ is complete. (Caution
that $G\backslash H^1$ is not a vector space, however.)
\end{lemma}

\begin{proof}
First, we establish that the orbits of $G$ are closed in $H^1$. The
orbit of $0$ is $0$.  Suppose $\phi\neq 0$, $\{ x_n \} \subset
\mathbb{R}^3$ and $\phi(\cdot-x_n)$ converges to $\psi$ in $H^1$.
Then we claim that $x_n$ converges.  Indeed, if not, then either
$x_n$ is unbounded and there is a subsequence $x_n$ such that
$|x_n|\to \infty$, or $x_n$ is bounded and there are two
subsequences $x_n\to x_0$ and $x_{n'}\to x_0'$.  In the first case,
we obtain that $\psi=0$ (by examining, for fixed $R>0$, the
convergence on $B(0,R)$), which implies $\phi=0$, a contradiction.
In the second case, we obtain that
$\phi(\cdot-x_0)=\phi(\cdot-x_0')$, only possible if $\phi=0$, a
contradiction.

Next, we verify that $d$ is a metric.  Suppose $d([\phi],[\psi])
=0$.  Then $\inf_{x_0\in\mathbb{R}^3}
\|\phi(\cdot-x_0)-\psi\|_{H^1}=0$, and thus $\psi$ is a point of
closure (in $H^1$) of the orbit of $\phi$.  But since the orbits are
closed, $\psi$ belongs to this orbit, and thus, $[\phi]=[\psi]$.
The triangle inequality is a straightforward exercise dealing with
infima, and symmetry is obvious.

Suppose $[\phi_n]$ is a Cauchy sequence; to show that it converges,
it suffices to show that a subsequence converges. We can pass to a
subsequence $[\phi_n]$ so that $d([\phi_n],[\phi_{n+1}])\leq
2^{-n}$.  Take $x_1=0$.  Construct a sequence $x_n$ inductively as
follows:  given $x_{n-1}$, select $x_n$ so that
$\|\phi_{n-1}(\cdot-x_{n-1})-\phi_n(\cdot-x_n)\|_{H^1}\leq
2^{-n+1}$.  Then $\phi_n(\cdot-x_n)$ is a Cauchy sequence in $H^1$,
and hence, converges to some $\phi$.  It is then clear that
$[\phi_n]\to [\phi]$ in $G\backslash H^1$.

It can be checked that for each $\phi\in H^1$ and $r>0$,
$\pi(B(\phi,r)) = B([\phi],r)$.  Therefore, the topology induced by
the metric $d$ on $G\backslash H^1$ is the quotient topology.
\end{proof}

The following two lemmas will reduce Prop. \ref{P:crit_compact}  to
proving that the set $\pi( \{\, u(\cdot,t) \, |\, t\in
[0,+\infty)\,\})$ is precompact in $G\backslash H^1$.

\begin{lemma}
\label{L:lifting}
Let $K$ be a precompact subset of $G\backslash H^1$. Assume
\begin{equation}
\label{lowerbound} \exists \;\eta >0 \text{ such that } \forall \;
\phi \in \pi^{-1}(K),\quad \eta\leq \|\phi\|_{H^1}.
\end{equation}
Then there exists $\tilde K$ precompact in $H^1$ such that $\pi(\tilde K) = K$.
\end{lemma}

\begin{proof}%
Let $B(0,1)$ be the unit ball in $\cR^3$. We first show by
contradiction that there exists $\eps>0$ such that for all $p$ in
$K$, there exists $\psi=\psi(p)\in \pi^{-1}(p)$ such that
\begin{equation}
\label{defpsi}
\|\psi(p)\|_{H^1(B(0,1))}\geq \eps.
\end{equation}
If not, there exists a sequence $\phi_n$ in $\pi^{-1}(K)$ such that
\begin{equation}
\label{absurde}
\sup_{x_0\in \cR^3} \| \phi_n(\cdot-x_0)\|_{H^1(B(0,1))}\leq \frac{1}{n}.
\end{equation}
The precompactness of $K$ implies, extracting a subsequence from
$\phi_n$ if necessary, that there exists $\phi\in H^1$ such that
$\pi(\phi_n)\to p$ in $G\backslash H^1$.  In other words, if $\phi$
is fixed in $\pi^{-1}(p)$, $\inf_{x_0\in \cR^3}
\|\phi_n(\cdot-x_0)-\phi\|_{H^1}$ tends to $0$ as $n$ tends to
infinity. Thus, one may find a sequence $x_n$ in $\cR^3$ such that
\begin{equation}
\label{convergence}
\|\phi_n(\cdot-x_n)-\phi\|_{H^1}\underset{n\rightarrow +\infty}{\longrightarrow} 0.
\end{equation}
Now, by \eqref{absurde}, for all $x_0\in \cR^3$, $
\|\phi_n(\cdot-x_0-x_n)\|_{H^1(B(0,1))}\leq \frac{1}{n}$. Hence, by
\eqref{convergence}, for all $x_0$, $\phi$ vanishes on $B(x_0,1)$.
But then $\phi=0$, which contradicts assumption \eqref{lowerbound},
concluding the proof of the existence of $x(\phi)$.

Let $\tilde{K}=\{ \psi(p)\;| \;  p\in K\, \}$, where $\psi(p)$
satisfies \eqref{defpsi}. Of course, $\pi(\tilde{K})=K$. By the
definition of $x(\phi)$,
\begin{equation}
\label{propertytK}
\forall \; \phi\in \pi^{-1}(K), \quad \|\phi\|_{H^1(B(0,1))}\geq \eps.
\end{equation}
Let us show that $\tilde K$ is precompact. Let $\phi_n$ be a
sequence in $\tilde K$. Then by the precompactness of $K$, there
exists (extracting subsequences) $\phi\in H^1$ and a sequence $x_n$
of $\cR^3$, such that
\begin{equation}
\label{limitvn}
\lim_{n\rightarrow +\infty} \|\phi_n(\cdot-x_n)-\phi\|_{H^1}=0.
\end{equation}
Note that $K$ being precompact, $\phi_n$ is bounded in $H^1$, thus,
we may assume (extracting again)
\begin{equation}
\label{limitnorm} \lim_{n\rightarrow+\infty} \|\phi_n\|_{H^1}=\ell
\in (0,+\infty).
\end{equation}
Let us show that $x_n$ is bounded. If not, we may assume that
$|x_n|\rightarrow +\infty$. By \eqref{propertytK} and
\eqref{limitnorm}, we have
\begin{equation*}
\limsup_{n\rightarrow
\infty}\|\phi_n(\cdot-x_n)\|_{H^1(B(0,|x_n|-1))}\leq \ell-\eps.
\end{equation*}
As $|x_n|\rightarrow \infty$, we conclude that $\|\phi\|_{H^1}\leq
\ell-\eps$, contradicting \eqref{limitnorm}. Therefore, $x_n$ is
bounded. Extracting if necessary, we may assume that $x_n$
converges, which shows by \eqref{limitvn} that $\phi_n$ converges.
This concludes the proof of the precompactness of $\tilde K$.
\end{proof}

\begin{lemma}
\label{L:cont-path}
Let $u$ be a global $H^1$ solution to \eqref{E:NLS}.  Suppose
$$\pi( \{\, u(\cdot,t) \, | \, t\in [0,+\infty) \,\})$$
is precompact in $G\backslash H^1$. Then there exists $x(t)$, a
continuous path in $\mathbb{R}^3$, such that
$$
\{ \, u(\cdot-x(t),t) \, | \, t\in [0,+\infty) \,\}
$$
is precompact in $H^1$.
\end{lemma}

\begin{proof}
By taking $K=\pi( \{\, u(\cdot,t) \, | \, t\in [0,+\infty) \,\})$ in
Lemma \ref{L:lifting}, we obtain $\tilde K$ precompact in $H^1$ such
that $\pi(\tilde K)=K$.  For each $N$, the map $u:[N,N+1]\to H^1$ is
uniformly continuous. Thus, for each $N$, there exists $\delta_N>0$
such that if $t,t'\in [N,N+1]$ and $|t-t'|\leq \delta_N$, then
$\|u(t,\cdot)-u(t',\cdot)\|_{H^1} \leq 1/N$.  Let $t_n$ be the
increasing sequence of times $\to +\infty$ defined to include evenly
spaced elements with density $\delta_N$ in $[N,N+1]$ for each $N$.
(Thus, $t_n$ is an increasing sequence with possibly more elements
per unit interval as we move out to $+\infty$).  For each $n$,
select $x(t_n)\in \mathbb{R}^3$ such that $u(\cdot-x(t_n),t_n)\in
\tilde K$.  Now define $x(t)$ to be the continuous function that
connects $x(t_n)$ to $x(t_{n+1})$ by a straight line in
$\mathbb{R}^3$.

We claim that $\{ \, u(\cdot-x(t),t)\,  | \, t\in [0,+\infty)\, \}$
is precompact in $H^1$.  Indeed, let $s_k$ be a sequence in
$[0,+\infty)$.  Then there exists a subsequence (also labeled $s_k$)
such that either $s_k$ converges to some finite $s_0$ or $s_k\to
+\infty$.  In the first case, $u(\cdot-x(s_k),s_k)\to
u(\cdot-x(s_0),s_0)$ by the continuity of $u(t)$ and $x(t)$.  In the
second case, for each $k$, obtain the unique index $n(k)$ such that
$t_{n(k)-1} \leq s_k <t_{n(k)}$.   By the precompactness of $\tilde
K$, we can pass to a subsequence (in $k$) such that both
$u(\cdot-x(t_{n(k)-1}), t_{n(k)-1})$ and
$u(\cdot-x(t_{n(k)}),t_{n(k)})$ converge in $H^1$.  By the density
of the $t_n$ sequence and uniform continuity of $u$, we obtain that
$u(\cdot-x(t_{n(k)-1}), t_{n(k)})$ converges and that it suffices to
show that $u(\cdot-x(s_k),t_{n(k)})$ has a convergent subsequence.
But since both $u(\cdot-x(t_{n(k)-1}), t_{n(k)})$ and
$u(\cdot-x(t_{n(k)}), t_{n(k)})$ converge, we have that
$x(t_{n(k)-1})-x(t_{n(k)})$ converges.  Recall that $x(s_k)$ lies on
the line segment joining $x(t_{n(k)-1})$ and $x(t_{n(k)})$, and
thus, $x(s_k)-x(t_{n(k)-1})$ converges (after passing to a
subsequence). Hence, $u(\cdot-x(s_k),t_{n(k)})$ converges in $H^1$.
\end{proof}

Thus, to prove Prop. \ref{P:crit_compact}, it suffices to prove that
\begin{equation}
\label{E:projection}
\pi(\{ \, u(\cdot, t) \, | \, t\in [0,+\infty) \, \})
\end{equation}
is precompact in $G\backslash H^1$.  Since $G\backslash H^1$ is complete, if we assume that \eqref{E:projection} is not precompact in $G\backslash H^1$, then there exists a sequence $\{ [u(t_n)]\}$ in $G\backslash H^1$ and $\eta>0$ such that $d([u(t_n)],[u(t_{n'})]) \geq \eta$, or equivalently, \eqref{E:negation} in the proof of Prop \ref{P:crit_compact} holds.

\end{document}